\documentclass[12pt]{amsart}
\newcounter{defcounter}
\usepackage{amssymb,amsmath,amscd,graphicx, enumitem,
latexsym,amsthm}
\usepackage{amssymb,latexsym,eufrak,amsmath,amscd,graphicx, amsxtra}
  \usepackage[all]{xy}
  \usepackage{pdfsync}
  \usepackage[mathscr]{euscript}

\theoremstyle{plain}
\newtheorem{theorem}{Theorem}

\newtheorem{proposition}[theorem]{Proposition}
\newtheorem{corollary}[theorem]{Corollary}
\newtheorem{lemma}[theorem]{Lemma}

\newtheorem{proposition.definition}[theorem]{Proposition/Definition}

\newtheorem{theoremalpha}{Theorem}

\newtheorem{corollaryalpha}[theoremalpha]{Corollary}
\newtheorem{propositionalpha}[theoremalpha]{Proposition}

\theoremstyle{definition}

\newtheorem{remark}[theorem]{Remark}
\newtheorem{example}[theorem]{Example}

\newcommand{\lra}{\longrightarrow}

\newcommand{\noi}{\noindent}
\newcommand{\PP}{\mathbf{P}}

\newcommand{\CC}{\mathbf{C}}

\newcommand{\Zero}[1]{\textnormal{Zeroes} ({#1})}
\newcommand{\OO}{\mathcal{O}}

\newcommand{\cH}{\mathcal{H}}
\newcommand{\II}{\mathcal{I}}
\newcommand{\UU}{\mathcal{U}}
\newcommand{\FF}{\mathcal{F}}

\newcommand{\frb}{\mathfrak{b}}

\newcommand{\frakm}{\mathfrak{m}}

\newcommand{\eps}{\varepsilon}

\newcommand{\rk} {\text{rank }}

\newcommand{\Image}{\textnormal{Im}}
\newcommand{\image}{\textnormal{im}}

\newcommand{\HH}[3]{H^{{#1}} \big( {#2} , {#3}
\big) }

\newcommand{\HHH}[3]{H^{{#1}} \Big( {#2} \, , \, {#3}
\Big) }
\newcommand{\HHHH}[4]{H_{{#1}}^{{#2}} \big( {#3} \, , \, {#4}
\big) }

\newcommand{\MI}[1]{\mathcal{J} \big( {#1}
\big) }

\newcommand{\ord}{\textnormal{ord}}
\newcommand{\codim}{\textnormal{codim}}

\newcommand{\pr}{\prime}

\newcommand{\Bl}{\text{Bl}}

\newcommand{\reg}{\textnormal{reg}}
\newcommand{\EEE}{\mathbf{E}}

\newcommand{\Sym}{\textnormal{Sym}}

\newcommand{\satt}[1]{{#1}^{\textnormal{sat}}}
\newcommand{\satdeg}{\textnormal{sat. \!deg}}
\newcommand{\shext}{\mathcal{E}\mathit{xt}}
\newcommand{\shom}{\mathcal{H}\mathit{om}}

\numberwithin{theorem}{section}

\begin{document}

\title{Saturation  bounds for smooth varieties}
 \author{Lawrence Ein}
 \address{Department of Mathematics, University of Illinois at Chicago, 851 South Morgan St., Chicago, IL  60607}
  \email{{\tt ein@uic.edu}}
  \thanks{Research of the first author partially supported by NSF grant DMS-1801870.}
 
\author{Huy T\`ai H\`a}
\address{Tulane University \\ Department of Mathematics \\
                6823 St. Charles Ave. \\ New Orleans, LA 70118, USA}
\email{tha@tulane.edu}
  \thanks{Research  of the second author partially supported by  Louisiana Board of Regents (grant \# LEQSF(2017-19)-ENH-TR-25).}
 
 \author{Robert Lazarsfeld}
  \address{Department of Mathematics, Stony Brook University, Stony Brook, New York 11794}
 \email{{\tt robert.lazarsfeld@stonybrook.edu}}
 \thanks{Research  of the third author partially supported by NSF grant DMS-1739285.}

\maketitle

 \section*{Introduction}
 
The purpose of this paper is to prove some saturation bounds for the ideals of non-singular complex projective varieties and their powers.

We begin with some background. Consider the polynomial ring  $S = \CC[x_0, \ldots, x_r]$  in $r+1$ variables, and fix homogeneous polynomials
\[f_0\, , \,  f_1 \, , \, \ldots \, , \, f_p\, \in \, S \ \ \text{with\ \  $\deg(f_i) = d_i$}.\] We assume that
$d_0  \ge  d_1  \ge   \ldots  \ge  d_p$, 
and we denote by 
\[ J \ = \ \big ( \, f_0 \, , \, f_1 \, , \, \ldots \, ,\, f_p \, \big) \ \subseteq \ S \]
the ideal that the polynomials span. Suppose now that $J$ is primary for the irrelevant maximal ideal $\frakm = (x_0,  \ldots, x_r)$, or equivalently that $\dim_\CC S/J  < \infty$. In this case $J$ contains all monomials of sufficiently large degree, and it is a classical theorem of Macaulay \cite[Theorem 7.4.1]{CM-S.P} that 
\begin{equation} \label{Macaulay.Eqn.1}
J_t \ = \ S_t \ \ \text{ for } \ \ t\,  \ge \, d_0 + \ldots + d_r - r. 
\end{equation}
Moreover this bound is (always) sharp when $p = r$. 
Although less well known,  a similar statement holds for powers of $J$:
\begin{equation} \label{Macaulay.Eqn.2}
(J^a)_t \ = \ S_t \ \ \text{ for } \ \ t\,  \ge \, ad_0 +d_1 + \ldots + d_r - r.\end{equation}
This again is always sharp when $ p = r$. 

It is natural to ask whether there are analogous results for more general homogeneous ideals $J$, in particular when 
\[    X \ =_{\text{def}} \ \Zero{J} \ \subseteq \ \PP^r \]
is a smooth complex projective variety. Of course if $J$ has non-trivial zeroes, then it does not contain any power of the maximal ideal. However if one interprets \eqref{Macaulay.Eqn.1} and \eqref{Macaulay.Eqn.2} as saturation bounds, then the question makes sense more generally. Specifically, recall that the \textit{saturation} of a homogeneous ideal $J$ is defined by
\begin{equation}  \satt{J}  \ = \ \big \{ \, f \in S \mid \frakm^k \cdot f \subseteq J     \text{ for some $k \ge 0$} \, \big \}. \notag \end{equation}
The quotient $ \satt{J} / J$ has finite length, and in particular
\[ ( \satt{J} )_t \ = \ J_t \ \ \text{ for } \ t \gg 0. \]
The least such integer $t$ is called the \textit{saturation degree} $\satdeg(J)$ of $J$. Observing that $\satt{J} = S$ if and only if $J$ is $\frakm$-primary,  statements \eqref{Macaulay.Eqn.1} and \eqref{Macaulay.Eqn.2} are equivalent to estimates for the saturation degrees of $J$ and $J^a$. So the problem becomes to   bound the saturation degree of an ideal in terms of the degrees of its generators.

It is instructive to consider some examples. Let $X \subseteq \PP^r$ be a hyperplane defined by a linear form $\ell \in S$, and set
\begin{equation} \label{Hyperplane.Example}  f_i \, = \, x_i^{d-1}\cdot \ell \ \  \ , \  \ \ J = (f_0, \ldots, f_r)\, \subseteq \, S. \end{equation}
Then $\satt{J} = (\ell)$, and it follows from Macaulay's theorem that
\[  \satdeg(J) \ = \ (r+1)(d-1) - r + 1 \ = \ (r+1)d - 2r, \]
which is very close to the bound \eqref{Macaulay.Eqn.1}. On the other hand, it is not the case that the saturation degree  of an arbitrary ideal is bounded linearly in the degrees of its generators. For instance, 
 the ideals \[  J \ = \ \big( x^d, y^d, xz^{d-1} - yw^{d-1} \big) \ \subseteq \ \CC[x,y,z,w] \]
considered by Caviglia \cite[Example 4.2.1]{Caviglia} have  $\satdeg(J) \approx d^2$.

Our first main result asserts that for ideals defining smooth varieties, the Macaulay bounds remain true without modification.
\begin{theoremalpha} \label{Intro.Sat.Deg.Thm}
As above, suppose that 
\[ J \ = \ \big ( \, f_0 \, , \, f_1 \, , \, \ldots \, ,\, f_p \, \big) \ \subseteq \ S \]
is generated by forms of degrees $d_0 \ge \ldots \ge d_p$, and assume that the projective scheme
\[    X \ =_{\text{def}} \ \Zero{J} \ \subseteq \ \PP^r \]
cut out by the $f_i$ is a non-singular complex variety. Then 
$ \satdeg(J)   \le  d_0 + \ldots + d_r - r$,  
and more generally
\begin{equation} \label{Intro.Thm.Equation} \satdeg(J^a) \ \le \ ad_0 + d_1 + \ldots + d_r - r. \end{equation}
\end{theoremalpha}
\noi (If $p < r$, one takes $d_{p+1} = \ldots = d_r = 0$.) We do not know whether the stated bound is best possible, but in any event it is asymptotically sharp. Indeed, if $J$ is the ideal considered in \eqref{Hyperplane.Example}, then the Theorem predicts that $\satdeg(J^a)\le (a+r)d -r$, whereas in fact $\satdeg(J^a) = (a+r)d - 2r$.

Given a reduced algebraic set $X \subseteq \PP^r$ denote by $I_X \subseteq S$ the saturated homogeneous ideal of $X$. Recall that the \textit{symbolic powers} of $I_X$ are
\[
I_X^{(a)} \ = \ \big \{ f \in S \mid \ord_x(f) \ge a  \text{ for general (or every) }  x\in X \, \big \}. 
\]
Evidently $I_X^a \subseteq I_X^{(a)}$, and there has been a huge amount of interest in recent years in understanding the connections between actual and symbolic powers (cf \cite{ELS}, \cite{Hochster.Huneke}, \cite{BocciHarbourne}, \cite{Dao.ea}). If $X$  is non-singular, then $I_X^{(a)} = \satt{(I_X^a)}$. Therefore Theorem \ref{Intro.Sat.Deg.Thm} implies
\begin{corollaryalpha}
Assume that $X \subseteq \PP^r$ is smooth, and that $I_X$ is generated in degrees $d_0 \ge d_1 \ge \ldots \ge d_p$. Then
\[   \big ( I_X^{(a)} \big)_t \ = \ ( I_X^a)_t  \ \ \text{ for } \, t \, \ge \, ad_0 +d_1 + \ldots + d_r -r. \]
\end{corollaryalpha}
\noi 
For example, suppose that $X \subseteq \PP^2$ consists of the three coordinate points, so that $I_X = (xy, yz, zx) \subseteq \CC[x,y,z]$.  The Corollary guarantees that $I_X^a$ and $I_X^{(a)}$ agree in degrees $\ge 2a + 2$, whereas in reality $\satdeg (I_X^a) = 2a$. So here again the statement is asymptotically but not precisely sharp.

In the case of finite sets,  results of Geramita-Gimigliano-Pitteloud \cite{GGP}, Chandler \cite{Chandler} and Sidman \cite{Sidman}   provide an alternative   bound that is often best-possible. Recall that a scheme $X \subseteq \PP^r$ is said to be $m$-regular in the sense of Castelnuovo--Mumford if its ideal sheaf $\II_X \subseteq \OO_{\PP^r}$ satisfies the vanishings:
\[
\HH{i}{\PP^r}{\II_X(m-i)} \  = \ 0 \ \ \text{ for \ } i > 0.  
\]
This is equivalent to asking that $I_X$ be generated in degrees $\le m$, that the first syzygies among minimal generators of $I_X$ appear in degrees $\le m+1$, the second syzygies in degrees $\le m+2$, and so on.\footnote{For saturated ideals, Castelnuovo--Mumford regularity of $I_X$ agrees with an algebraic notion of regularity introduced by Eisenbud and Goto \cite{Eisenbud.Goto} that we propose to call \textit{arithmetic regularity}. An arbitrary ideal $J \subseteq S$ is arithmetically $m$-regular if and only if $\satt{J}$ is $m$-regular and $\satdeg(J) \le m$. Given that we are interested in establishing bounds on saturation degree, unless otherwise stated we always refer to regularity in the geometric sense.}
The authors just cited show that if $X \subseteq \PP^r$ is an $m$-regular finite set, then
\[
\satdeg (I_X^a) \ \le \ am. 
\]
 This is optimal for the example of the three coordinate points in $\PP^2$. 
 
 Our second main result asserts  that the same statement holds when $\dim X = 1$. 
\begin{theoremalpha} \label{Regularity.Saturation.Bound.Curves}
Let $X \subseteq \PP^r$ be a smooth $m$-regular curve. Then
\[
\big( I_X^a\big)_t \ = \ \big( I_X^{(a)}\big)_t \ \ \text{for }   \ t \, \ge \, a m. 
\]
\end{theoremalpha}
\noi In fact, for the saturation bound  it suffices that the curve $X$ be reduced. The statement is optimal (for all $a$) for instance when $X \subseteq \PP^4$ is a rational normal curve. 
We also show that if $X \subseteq \PP^r$ is a reduced surface, then $\reg(\II_X^a) \le a \cdot \reg(\II_X)$. 
We do not know any examples where the analogous statements fail for smooth varieties of higher dimension. 

Returning to the setting of Theorem \ref{Intro.Sat.Deg.Thm}, the first and third authors showed with Bertram some years ago \cite{BEL}  that if   $X \subseteq \PP^r$ is a smooth complex projective variety of codimension $e$ cut out as a scheme by homogeneous polynomials of degrees $d_0 \ge \ldots \ge d_p$, then $\II_X^a$ is $(ad_0 + d_1 +  \ldots + d_{e-1} -e)$-regular in the sense of Castelnuovo-Mumford. Note however that this  does not address the questions of saturation required to control the arithmetic (Eisenbud--Goto) regularity of $I_X^a$.\footnote{In particular, the proof of Proposition 2.2 in \cite{AV} seems to be erroneous.}  In fact, one can view Theorem \ref{Intro.Sat.Deg.Thm} as promoting the results of \cite{BEL} to statements about arithmetic regularity:
 \begin{corollaryalpha} Assume that $J \subseteq S$ satisfies the hypotheses of Theorem \ref{Intro.Sat.Deg.Thm}. Then
 \[ \textnormal{arith. \!reg}(J^a) \ \le \ ad_0 + (d_1 + \ldots + d_r -r).\]
 \end{corollaryalpha}
\noi 
It is known (\cite{Kod}, \cite{CHK}) that if $J \subseteq S$ is an arbitrary homogeneous ideal then \[  \textnormal{arith. \!reg}(J^a) \ =  \ ad  + b  \ \text{ \ when  } \ a \gg 0, \] where $d$ is the maximal degree needed to generate a reduction of $J$ -- which coincides with  the generating degree of $J$ when it is equigenerated -- and $b$ is some constant. However computing the constant term $b$ has proven elusive, and  the Corollary gives a bound in the case at hand. 

The proofs of these results revolve around using  complexes of sheaves to study the image in $\HHHH{*}{0}{\PP^r}{\II_X^a} = \satt{(I_X^{a})}$ of the powers of the ideal spanned by generators of $I_X$ or $J$: this approach was inspired in part by geometrizing the arguments of Cooper and coauthors for codimenson two subvarieties in \cite{Cooper+}. Specifically, suppose that
\[    \eps : U_0 \, =_{\text{def}} \, \oplus \, \OO_{\PP^r}(-d_i) \lra \II_X  \]
is the surjective map of sheaves determined by generators of $I_X$ or $J$. If $X$ is $m$-regular, then this sits in an exact complex $U_\bullet$ of bundles:
\[
0 \lra U_{r-1} \lra U_{r-2} \lra \ldots \lra U_1 \lra U_0 \overset{\eps}  \lra \II_X \lra 0
\] where $\reg(U_i) \le m + i$. Weyman \cite{Weyman} (see also \cite{Tchernev}) constructs a new complex $L_\bullet = \Sym^a(U_\bullet)$  that takes the form
\[
\ldots \lra L_2 \lra L_1 \lra S^a(U_0) \lra \II_X^a \lra 0
\]
where $\reg(L_i) \le am + i$. This complex is exact only off $X$, but as in \cite{GLP} when $\dim X = 1$ one can still read off the surjectivity of 
\[
\HH{0}{\PP^r}{S^a (U_0)(t)} \lra \HH{0}{\PP^r}{\II_X^a(t)}  
\]
for $t \ge am$. This gives Theorem \ref{Regularity.Saturation.Bound.Curves}. 

Turning to Theorem  \ref{Intro.Sat.Deg.Thm}, a natural idea is to start with the Koszul complex
\[ \ldots \lra \Lambda^3 U_0 \lra \Lambda^2 U_0 \lra U_0 \lra \II_X \lra 0.\]
As established by Buchsbaum--Eisenbud \cite{Buchsbaum.Eisenbud}, this determines a new complex
\[
\ldots \lra S^{a, 1^2}(U_0) \lra S^{a,1}(U_0) \lra S^a(U_0) \lra \II_X^a \lra 0, \tag{*}
\]
where $S^{a, 1^k}(U_0)$ denotes the Schur power of $U_0$ corresponding to the Young diagram $(a, 1^{k})$. We observe that 
\[   \reg \big(S^{a, 1^{i}}(U_0) \big) \ \le  \ ad_0 + d_1 + \ldots + d_i, \]
so if (*) were exact then the statement of the Theorem would follow immediately. Unfortunately (*) is exact only if $X$ is a complete intersection, but by blowing up $X$ this construction yields an exact complex whose cohomology groups one can control with some effort. At the end of the day, the computation boils down to using Kodaira--Nakano vanishing on $X$ to prove 
a vanishing statement for  symmetric powers of the normal bundle to $X$ in $\PP^r$:
\begin{propositionalpha} \label{NB.Vanishing.Prop}
Let $X \subseteq \PP^r$ be a smooth complex projective variety, and denote by $N = N_{X/\PP^r}$ the normal bundle to $X$ in $\PP^r$. Then
\[ \HHH{i}{X}{S^k N \otimes \det N \otimes \OO_X(\ell) } \ = \ 0  \ \ \text{ for } i > 0
\]
and every $k \ge 0$, $\ell \ge -r$. 
\end{propositionalpha}
\noi (Similar but slightly different vanishings were established by Schneider and Zintl in \cite{Schneider.Zintl}.)
 We hope that some of these ideas may find other applications in the future.\footnote{We remark that some of the auxiliary results appearing here -- for example the Proposition just stated -- were known to the first and third authors some years ago in connection with their work on \cite{BEL}. However they were put aside in favor of the simpler arguments  with vanishing theorems that eventually appeared in that paper. }
  
 The paper is organized as follows. The first section is devoted to Theorem \ref{Regularity.Saturation.Bound.Curves}. We collect in \S 2 some preliminary results towards the Macaulay-type bounds. Specifically, we discuss the Buchsbaum--Eisenbud powers of Koszul complexes, the computation of some push-forwards from a blowing-up, and Proposition \ref{NB.Vanishing.Prop}. The proof of Theorem \ref{Intro.Sat.Deg.Thm} occupies \S 3. We work throughout over the complex numbers.
 
 We are grateful to  Sankhaneel Bisui, David Eisenbud, Elo\'isa Grifo and Claudia Miller for valuable remarks and correspondence.

\numberwithin{equation}{section}
\section{Saturation and regularity}

The present section is devoted to the proof of Theorem \ref{Regularity.Saturation.Bound.Curves} from the Introduction.

We start with some general remarks. Let $X \subseteq \PP^r$ be a  complex projective variety or scheme, with ideal sheaf $\II_X \subseteq \OO_{\PP^r}$  and homogeneous ideal $I_X \subseteq S$.  
Denote by $U_\bullet$  the locally free resolution of $\II_X$ obtained by sheafifying a minimal graded free resolution of $I_X$:
\begin{equation} \label{m-reg.Resoln.I}
0 \lra U_{r} \lra U_{r-1} \lra \ldots \lra U_1 \lra U_0 \overset{\eps}  \lra \II_X \lra 0.  \end{equation}
  Thus 
each $U_i$ is a direct sum of line bundles, and we recover the original resolution as the the complex $\HHHH{*}{0}{\PP^r}{U_\bullet}$ obtained from $U_\bullet$ by taking global sections of all twists. 

Consider now the surjective homomorphism of sheaves
\[ S^a(\eps) \, : \, S^a U_0 \lra \II_X^a . \]
For any   $t \ge 0$  one has
\[ \HH{0}{\PP^r}{\II_X^a(t)} \ = \ \left( \satt{\left( I_X^a \right)} \right)_t. \]
On the other hand, the fact that  $U_0$ is constructed from minimal generators of $I_X$ implies that 
\[
\Image \Big( \HH{0}{\PP^r}{S^a (U_0) (t) } \lra \HH{0}{\PP^r}{\II_X^a(t)} \Big) \ = \  \big( I_X^a)_t .
\]
Therefore
\begin{lemma} \label{Surjectivity.Suffices.Lemma} The degree $t$ pieces of $I_X^a$ and $\satt{(I_X^a)}$ coincide if and only if the homomorphism
\[
\HH{0}{\PP^r}{S^a (U_0) (t) } \lra \HH{0}{\PP^r}{\II_X^a(t)} 
\]
determined by $S^a(\eps)$ is surjective.  \qed
\end{lemma}
\noi The plan is to study $S^a(\eps)$ by realizing it as the last map of a complex $S^a( U_\bullet)$.

Specifically, consider a smooth variety $M$, a subvariety $X \subseteq M$, and a locally free resolution $U_\bullet$ of $\II_X \subseteq \OO_M$ as above:
\begin{equation} \label{m-reg.Resoln.II}
0 \lra U_{r} \lra U_{r-1} \lra \ldots \lra U_1 \lra U_0 \overset{\eps}  \lra \II_X \lra 0.  \end{equation}
As explained by Weyman \cite{Weyman} and Tchernev \cite{Tchernev}, $U_\bullet$ determines  for fixed $a \ge 1$ a new complex $L_\bullet = S^a(U_\bullet)$ having the shape
\begin{equation} \label{Weyman.Complex.1}
\xymatrix @C=18pt{
\ldots \ar[r] & L_4 \ar[r] &L_3 \ar[r] & { \begin{matrix} S^{a-2}U_0 \otimes \Lambda^2 U_1 \\ \oplus \\ S^{a-1}U_0 \otimes U_2 \end{matrix}} \ar[r] &S^{a-1}U_0 \otimes U_1 \ar[r] & S^aU_0\ar[r]& \II_X^a  \ar[r] & 0.
}
\end{equation}
The last map on the right is $S^a(\eps)$, and the homomorphism  $S^{a-1}U_0 \otimes U_1 \lra S^a U_0$ is the natural one arising as the composition
 \[
 S^{a-1}U_0 \otimes U_1 \lra S^{a-1}U_0 \otimes U_0 \lra S^a U_0.
 \] The  $L_i$ are determined by setting
 \begin{equation} \label{Weyman.Complex.2}
 C^k (U_j) \ = \ \begin{cases} \ S^k U_j \ & \text{if $j$ is even} \\ \ \Lambda^k U_j \ & \text{if $j$ is odd} \end{cases},
 \end{equation}
and then taking
\begin{equation} \label{Weyman.Complex.3}
L_i\ = \ \bigoplus_{
\substack{k_0 + \ldots + k_r = a \\ k_1 + 2k_2 + \ldots + rk_r = i}} C^{k_0}(U_0) \otimes C^{k_1}(U_1) \otimes \ldots \otimes C^{k_r}(U_r).
\end{equation}

It follows from \cite[Theorem 1]{Weyman} or \cite[Theorem 2.1]{Tchernev} that:
\begin{equation} \label{Exact.Away.From.X.Equation}
\text{The  complex } \eqref{Weyman.Complex.1}
  \text{ is exact away from $X$}. 
\end{equation}
In general one does not expect exactness at  points of $X$, but when $X$ is  smooth  the right-most terms at least  are well-behaved:
\begin{lemma} \label{Right.Hand.Exactness.Weyman.Complex}
Assume that $X$ is non-singular. Then the sequence
\[ S^{a-1}U_0 \otimes U_1 \lra S^a U_0 \lra \II_X^a \lra 0\]
is exact. 
\end{lemma}
\begin{proof} The question being local, we can work over the local ring $\OO = \OO_{M,x}$ of $M$ at a point $x \in X$. Since $X$ is smooth,  $\II = \II_{X, x} \subseteq \OO$ is  generated by a regular sequence of length $e = \codim \, X$. Thus $\II$ has a minimal presentation 
\[ \Lambda^2 \UU \lra \UU \lra \II \lra 0 
\]
given by the beginning of a Koszul complex, where $\mathcal U = \OO^e$ is a free module of rank $e$. Here  one checks by hand the exactness of 
\[
S^{a-1}\UU \otimes \Lambda^2 \UU \lra S^a \UU \lra \II^a \lra 0.
\]  
(Compare Proposition  \ref{Koszul.Complex.Power}
 below.) An arbitrary free presentation of $\II$   then has the form
\[    \Lambda^2 \UU \oplus \mathcal{A} \oplus \mathcal{B} \lra \UU \oplus \mathcal{A} \lra \II \lra 0,\]
where $\mathcal{A}$ is a free module mapping to zero in $\II$,  $\mathcal{B}$ is a free module mapping to zero in $\UU \oplus \mathcal{A}$, and the left-hand map is the identity on $\mathcal{A}$.  It suffices to verify the exactness of
\[
S^{a-1}\big( \UU \oplus \mathcal{A} \big)  \otimes \big( \Lambda^2 \UU \oplus \mathcal{A} \big) \lra S^a \big( \UU \oplus \mathcal{A}\big) \lra \II^a \lra 0, 
\]
and this is clear upon writing  $S^a \big( \UU \oplus \mathcal{A}\big)=S^a \UU \, \oplus \, \mathcal{A} \otimes 
S^{a-1}\big( \UU \oplus \mathcal{A} \big)$. \end{proof}

With these preliminaries out of the way, we now prove (a slight strengthening of) Theorem \ref{Regularity.Saturation.Bound.Curves} from the Introduction.
\begin{theorem} \label{Reduced.Curve.Theorem} Let $X \subseteq \PP^r$ be a reduced $($but possibly singular$)$ curve, and assume that $X$ is $m$-regular in the sense of Castelnuovo--Mumford. Denote by $I_X \subseteq S$ the homogeneous ideal of $X$. Then
\[
\satdeg( I_X^a) \ \le \ a m.
\]
\end{theorem}
\begin{proof}
The $m$-regularity of $X$ means that we can take a resolution $U_\bullet$ of $\II_X$ as in \eqref{m-reg.Resoln.I} where $U_i$ is a direct sum of line bundles of degrees $\ge -m -i$, ie $\reg(U_i) \le m + i$. Consider the resulting Weyman complex $L_\bullet = S^a(U_\bullet)$:
\[
\lra L_3 \lra L_2 \lra L_1 \lra L_0 \lra \II_X^a  \lra 0,  \tag{*} \]
the last map being the surjection $S^a(\eps) : L_0 = S^a U_0 \lra \II_X $. In view of Lemma \ref{Surjectivity.Suffices.Lemma}, the issue is to establish the surjectivity of the homomorphism
\[  \HH{0}{\PP^r}{L_0(t)} \lra \HH{0}{\PP^r}{\II_X^a(t)}  \tag{**} \]
for $t \ge am$. To this end, observe first from \eqref{Weyman.Complex.2} and 
\eqref{Weyman.Complex.3} that
\[  \reg(L_i) \ \le \ am + i. \]
Consider next the homology sheaves $\cH_i = \cH_i(L_\bullet \lra \II_X^a)$ of the augmented complex (*).  (So for $i = 0$ we understand  $\cH_0 = \ker ( L_0 \lra \II_X^a) / \Image(L_1 \lra L_0).)$   Thanks to \eqref{Exact.Away.From.X.Equation}, these are all supported on the one-dimensional set $X$. Moreover  it follows from Lemma \ref{Right.Hand.Exactness.Weyman.Complex} that $\cH_0$ is supported on the finitely many singular points of $X$. Therefore the required surjectivity (**) is a consequence of the first statement of the following Lemma.
 \end{proof}
 
 \begin{lemma} Consider a complex $L_\bullet$ of coherent sheaves on $\PP^r$ sitting in a diagram
 \begin{equation}\label{Chopping.Lemma.Eqn} \ldots \lra L_3 \lra L_2 \lra L_1 \lra L_0\overset{\eps} \lra \FF  \lra 0,   \end{equation}
 and denote by $\cH_i = \cH_i(L_\bullet \lra \FF)$ the $i^{\text{th}}$ homology sheaf of the augmented complex \eqref{Chopping.Lemma.Eqn}.\footnote{So as above,  the group of zero-cycles used to compute  $\cH_0$ is $\ker (\eps)$.}
Assume that $\eps$ is surjective, and let $p$ be an integer with the property that  $L_i$ is $(p + i)$-regular  for every $i$.  
\begin{enumerate}
\item [$(i)$]  If each  $\cH_i$  is supported on a set of dimension $\le i$, then the homomorphism
\[  \HH{0}{\PP^r}{L_0(t) } \lra \HH{0}{\PP^r}{\FF(t)} \]
is surjective for $t \ge p$. 
\vskip 5pt
\item[$(ii)$] If each $\cH_i$ is supported on a set of dimension $\le i + 1$, then $\FF$ is $p$-regular. 
\end{enumerate}
 \end{lemma}
\begin{proof}
 This is established by chopping $L_\bullet$ into short exact sequences in the usual way and chasing through the resulting diagram. (Compare \cite[B.1.2, B.1.3]{PAG}, but note that the sheaf $\cH_0$ there should refer to the augmented complex, as above.)  \end{proof}

We conclude this section by observing that the same argument proves that  Castelnuovo--Mumford regularity of surfaces behaves submultiplicatively in powers. For curves, this has been known for some time \cite{Chandler}, \cite{Sidman}.
\begin{proposition}
Let $X \subseteq \PP^r$ be a reduced $($but possibly singular$)$ surface, and denote by $\II_X \subseteq \OO_{\PP^r}$ the ideal sheaf of $X$. If $\II_X$ is $m$-regular, then $\II_X^a$ is $am$-regular.
\end{proposition}
\begin{proof}[Sketch of Proof.] One argues just as in the  proof of Theorem \ref{Reduced.Curve.Theorem}, reducing to  statement (ii) of the previous Lemma. 
\end{proof}

\newcommand{\Schur}[2]{S^{{#1},1^{#2}}}

\section{Macaulay-type bounds: preliminaries}

This section is devoted to some preliminary results that will be used in the proof of Theorem \ref{Intro.Sat.Deg.Thm} from the Introduction. In the first subsection, we discuss symmetric powers of a Koszul complex. The second is devoted to the computation of some direct images from a blow-up. Finally \S  \ref{Vanishing.Theorem.Normal.Bundles.Subsection} gives the proof of Proposition \ref{NB.Vanishing.Prop} form the Introduction.

\subsection{Powers of Koszul complexes} \label{Powers.of.Koszul.Subsection}

In this subsection we review the construction of symmetric powers of a Koszul complex. In the local setting this (and much more) appears in the paper \cite{Buchsbaum.Eisenbud} of Buchsbaum and Eisenbud, and it was revisited by Srinivasan in \cite{Srinivasan}. However for the convenience of the reader we give here a quick sketch of the particular facts we require. We continue to  work  over the complex numbers.

Let $M$ be a smooth algebraic variety, and let $V$ be a vector bundle of rank $e$ on $M$. Fix integers $a, k \ge 1$. We denote by $\Schur{a}{k-1}(V)$ the Schur power of $V$ corresponding to the partition $(a, 1, \ldots, 1)$ ($k-1$ repetitions of $1$). It follows from Pieri's rule that
\begin{equation} \label{Schur.Equation}
\begin{aligned}\Schur{a}{k-1}(V) \ &= \ \ker \Big( \Lambda^{k-1}V \otimes S^a V \lra \Lambda^{k-2}V \otimes S^{a+1}V \Big) \\ &= \ \image \Big( \Lambda^k V \otimes S^{a-1}V \lra  \Lambda^{k-1} V \otimes  S^a V \Big). 
\end{aligned}
\end{equation}

\begin{remark} [Properties of $\Schur{a}{k-1}(V)$] \label{Properties.of.Schur.Power} We collect some useful observations concerning this Schur power.
\begin{enumerate}
\item [(i).] If $ k = 1$ then $\Schur{a}{k-1}(V)  = S^aV$, while if $a = 1$ then $\Schur{a}{k-1}(V)  = \Lambda^k V$. Moreover
\[   \Schur{a}{k-1}(V) \, = \, 0 \ \ \text{ when } k > \rk V. \]
\vskip 5pt
\item[(ii).] The bundle $\Schur{a}{k-1}(V)$ is actually a summand of $S^{a-1}V \otimes \Lambda^k V$. In fact,  Pieri shows that
\[
S^{a-1}V \otimes \Lambda^k V \ = \ \Schur{a}{k-1}(V)\,  \oplus \, \Schur{a-1}{k}(V).
\]
\vskip 5pt
\item[(iii).] If $L$ is a line bundle on $M$, then it follows from \eqref{Schur.Equation} or (ii) that 
\[ \Schur{a}{k-1} ( V \otimes L ) \ = \  \Schur{a}{k-1} ( V) \, \otimes \, L^{\otimes a + k -1}.\]
\item[(iv).] Suppose that $M = \PP^r$ and 
\[    V \ = \  \OO_{\PP^r}(-d_0 )  \oplus \, \ldots \, \oplus \,  \OO_{\PP^r}(-d_p) \]
with $\ d_0 \ge \ldots \ge d_p.$ Then it follows from (ii) that  $\Schur{a}{k-1}(V)$ is a direct sum of line bundles of degrees $\ge \, -(ad_0 + d_1 + \ldots + d_{k-1})$, and moreover a summand of this degree appears. In other words,
\[   \reg \big( \, \Schur{a}{k-1}(V) \, ) \ = \ ad_0 + d_1 + \ldots + d_{k-1}.\]
\end{enumerate}
\end{remark}

One can also realize $\Schur{a}{k-1}(V)$ geometrically, \`a la Kempf \cite{Kempf}.
\begin{lemma} \label{Kempf.Type.Lemma} 
Let $ \pi : \PP(V) \lra M $  be the projective bundle of one-dimensional quotients of $V$, and denote by $F$ the kernel of the canonical quotient $\pi^* V \lra \OO_{\PP(V)}(1)$, so that $F$ sits in the short exact sequence
\[
0 \lra F \lra \pi^* V \lra \OO_{\PP(V)}(1) \lra 0 \tag{*}
\]
of bundles on $\PP(V)$. Then
\[ \Schur{a}{k-1}(V) \ = \ \pi_* \Big( \, \Lambda^{k-1}F \otimes \OO_{\PP(V)}(a) \, \Big). \]
\end{lemma} 
\begin{proof} In fact, (*) gives rise to a long exact sequence
\small
\[
0 \lra \Lambda^{k-1} F \otimes \OO_{\PP(V)}(a) \lra \Lambda^{k-1} (\pi^* V) \otimes \OO_{\PP(V)}(a) \lra \Lambda^{k-2} (\pi^* V) \otimes \OO_{\PP(V)}(a+1) \lra \ldots \ . 
\]
\normalsize
The assertion follows from \eqref{Schur.Equation} upon taking direct images.
\end{proof}
 
 Now suppose given a map of bundles
 \begin{equation} \label{cosection}
   \eps : V \lra \OO_M  \end{equation}
whose image is the ideal sheaf $\II \subseteq \OO_M$ of a subscheme $Z \subseteq X$:  equivalently, $\eps$ is dual to a section $\OO_M \lra V^*$ whose zero-scheme is $Z$.  We allow the possibility that $\eps$ is surjective, in which case $\II = \OO_M$ and $Z = \varnothing$.

If $Z$ has the expected codimension $e = \rk(V)$, then $\II$ is resolved by the Koszul complex associated to $\eps$. The following result of Buchsbaum and Eisenbud gives the  resolution of powers of  $\II$.
\begin{proposition} [{\cite[Theorem 3.1]{Buchsbaum.Eisenbud}, \cite[Theorem 2.1]{Srinivasan}}] \label{Power.Koszul.Complex.Proposition}
Fix $a \ge 1$. Then $\eps$ determines a complex 
\begin{equation} \label{Koszul.Complex.Power}
\xymatrix@C=30pt{
\ldots  \ar[r]  &\Schur{a}{2}(V) \ar[r] &S^{a,1}(V)  \ar[r]     &S^a V     \ar[r]^{S^a(\eps)}  & \II^a \ar[r]  & 0  
}
\end{equation}
of vector bundles on $M$.
This complex is exact provided that either $\eps$ is surjective, or that $Z$ has codimension $= \rk(V)$. 
\end{proposition}
\noi Observe from \ref{Properties.of.Schur.Power} (i) that this complex has the same length as the Koszul complex of $\eps$. 

\begin{proof} Returning to the setting of Lemma \ref{Kempf.Type.Lemma}, denote by $\tilde{\eps} : F \lra \OO_{\PP(V)}$ the composition of the inclusion $F \hookrightarrow \pi^*V$ with $\pi^*\eps : \pi^* V \lra \pi^* \OO_M$. The zero-locus of $\tilde{\eps}$ defines the natural embedding of $\PP(\II)$ in $\PP(V)$. Now consider the Koszul complex of $\tilde \eps$. After twisting by $\OO_{\PP(V)}(a)$ this has the form:
\[
\ldots \lra \Lambda^2 F \otimes \OO_{\PP(V)}(a) \lra F \otimes \OO_{\PP(V)}(a) \lra \OO_{\PP(V)}(a) \lra \OO_{\PP(\II)}(a) \lra 0.  \tag{*}
\]
In view of Lemma  \ref{Kempf.Type.Lemma}, \eqref{Koszul.Complex.Power} arises by taking direct images. If $\eps$ is surjective, or defines a regular section of $V^*$, then the Koszul complex (*) is exact. Since the higher direct images of all the terms vanish, (*) pushes down to an exact complex. Furthermore, in this case $\pi_* \OO_{\PP(\II)}(a) = \II^a$ (cf \cite[Theorem IV.2.2]{Fulton.Lang}), and the exactness of \eqref{Koszul.Complex.Power} follows. \end{proof}

\begin{example} [Macaulay's Theorem]
Suppose as in the Introduction that $f_0, \ldots, f_p \in    \CC[x_0, \ldots, x_r]$ are homogeneous polynomials of degrees $d_0 \ge \ldots \ge d_p$ that generate a finite colength ideal $J$. This gives rise to a surjective map
\[   V \ = \ \oplus \, \OO_{\PP^r} (-d_i) \lra \OO_{\PP^r} \lra 0\]
of bundles on projective space. Keeping in mind Remark \ref{Properties.of.Schur.Power} (iv), Macaulay's statements \eqref{Macaulay.Eqn.1}
 and \eqref{Macaulay.Eqn.2}
follow by looking at the cohomology of the resulting complex \eqref{Koszul.Complex.Power}. When $p = r$ this complex has length $r+1$, so one can also read off the non-surjectivity of 
\[  \HH{0}{\PP^r}{S^a V (t)} \lra  \HH{0}{\PP^r}{\OO_{\PP^r}(t)} \]
when $t < ad_0+ d_1 + \ldots + d_r -r$. 
\end{example}

\begin{example} [Complete intersection ideals] Suppose that $Z\subseteq \PP^r$ is a complete intersection of dimension $\ge 0$. Applying Theorem \ref{Koszul.Complex.Power}
to the Koszul resolution of its homogeneous ideal $I_Z$, one sees that $I_Z^a$ is saturated for every $a\ge 1$. This is a result of Zariski.
\end{example}

\subsection{Push-forwards from a blowing up}
\label{Pushfowards.from.Blowup.Subsection}

We compute here the direct images of multiples of the exceptional divisor under the blowing-up of a smooth subvariety.

Consider then a smooth  variety $M$ and a non--singular subvariety $X \subseteq M$ having codimension $e \ge 2$ and ideal sheaf $\II = \II_X \subseteq \OO_M$. We consider the blowing-up 
\[  \mu : M^\pr = \Bl_X(M) \lra  M \]
of $M$ along $X$.  Write $\EEE \subseteq  M^\pr$ for the exceptional divisor of $M^\pr$, so that $\II \cdot \OO_{M^\pr} = \OO_{M^\pr}(-\EEE)$. 
We recall that if $a>0$ then
\begin{equation} \label{BU.Eqn.1}
\mu_* \OO_{M^\pr}(-a\EEE) \ = \ \II^a \ \ \text{and } \ \ R^j \mu_* \OO_{M^\pr}(-a\EEE) \, = \, 0 \ \text{for } j > 0.
\end{equation}

The following Proposition gives the analogous computation for positive multiples of $\EEE$.
\begin{proposition} \label{Blowup.Pushforward.Proposition}
Fix $a > 0$. Then
\begin{equation} \label{Pushforward.Ext.Equation}
R^j \mu_* \OO_{M^\pr}(a\EEE) \ = \ \shext^j_{\OO_M}\Big( \II^{a-e+1} \, , \, \OO_M \Big).\footnote{When $0 < a < e-1$ we take $\II^{a-e+1} = \OO_M$.}
\end{equation}
In particular, $\mu_* \OO_{M^\pr}(a\EEE) = \OO_M$, $R^j \mu_* \OO_{M^\pr}(a\EEE) = 0$ if $j \ne 0, e-1$, and
\[
R^{e-1}\mu_* \OO_{M^\pr}(a\EEE) \ = \  \shext^{e-1}_{\OO_M}\big( \II^{a-e+1} \, , \, \OO_M \big). \]
\end{proposition} 

\begin{proof}[Proof of Proposition \ref{Blowup.Pushforward.Proposition}]
This is a consequence of duality for $\mu$, which asserts that
\[  
R\mu_* \, R\,\shom_{\OO_{M^\pr}} \big( \FF \, , \, \omega_\mu \, \big) \ = \  R\, \shom_{\OO_M} \big ( \, R\mu_* \FF \, , \OO_M \, \big) \tag{*}
\]
for any sheaf $\FF$ on $M^\pr$, where $\omega_\mu$ denotes the relative dualizing sheaf for $\mu$ (\cite[(3.19) on page 86]{Huybrechts}).   We apply this with 
\[  \FF \ = \ \OO_{M^\pr}\big( \, (e-1-a)\EEE\, ). \]
Then 
$R \mu_* \FF =  \II^{a-e+1}$
thanks to  \eqref{BU.Eqn.1} (and a direct computation when  $0 < a < e-1$), and $\omega_\mu = \OO_{M^\pr}\big( (e-1)\EEE \big)$. Therefore the first assertion of the Proposition follows from (*). The vanishing of $\shext^j_{\OO_M}(\II^{a-e+1}, \OO_M)$ for $j \ne 0, e-1$ follows from the perfection of powers of the ideal of a smooth variety (which in turn is a consequence eg of  Proposition \ref{Power.Koszul.Complex.Proposition}). 
\end{proof}

\begin{remark} [Generalization to multiplier ideal sheaves] Let $\frb \subseteq \OO_M$ be an arbitrary ideal sheaf, and let $\mu : M^\pr \lra M$ be a log resolution of $\frb$, with $\frb \cdot \OO_{M^\pr} = \OO_{M^\pr}(-\EEE)$. A completely parallel argument shows that for $a > 0$:
\[
R^j \mu_* \OO_{M^\pr}(a\EEE) \ = \ \shext^j_{\OO_M}\big( \MI{\frb^a} \, , \, \OO_M \big),\]
where $\MI{\frb^a}$ is the multiplier ideal of $\frb^a$. The formula \eqref{Pushforward.Ext.Equation} is a special case of this. 
\end{remark}

\begin{corollary} \label{Filtration.of.Push.Forwards}
Continuing to work in characteristic zero, fix  $a \ge 1$ and   denote by $N = N_{X/M}$ the normal bundle to $X$ in $M$.  If $a \le e-1$, then 
\[ R^{e-1} \mu_* \, \OO_{M^\pr}(a\EEE)  \ = \  0. \]    If $a \ge e$, then
$
R^{e-1} \mu_* \, \OO_{M^\pr}(a\EEE)  
$
has a filtration with successive quotients
\[
S^k N \otimes \det N \ \ \text{ for }   \ 0 \, \le \, k \, \le \, a-e.
\]
\end{corollary}
\begin{proof} The first statement follows directly from the previous Proposition. For the second, 
recall first that if $E$ is any locally free $\OO_X$-module, then  -- $X$ being non-singular of codimension $e$ in $M$ -- 
\[   \shext_{\OO_M}^{e} \big ( \, E \, , \, \OO_M \, \big) \ = \ E^* \otimes \det  N, \]
while all the other $\shext^j$ vanish. The claim then follows from Proposition \ref{Pushforward.Ext.Equation}
 using the exact sequences
\[ 0 \lra \II^{k+1} \lra \II^k \lra S^k N^* \lra 0 \]
together with the isomorphism $\big (S^k(N^*)\big)^* = S^k N$ valid in characteristic zero.
 \end{proof}

\begin{remark}
Recalling that $\EEE = \PP(N^*)$, one can inductively prove the Corollary directly, circumventing Proposition \ref{Pushforward.Ext.Equation}, by pushing forward the exact sequences
\[   0 \lra  \OO_{M^\pr}\big((k-1)\EEE\big) \lra  \OO_{M^\pr}\big(k 
\EEE\big) \lra \OO_{\EEE}(k\EEE)\lra 0. \]
However it seemed to us that the Proposition may be of independent interest. 
\end{remark}

\subsection{A vanishing theorem for normal bundles}
\label{Vanishing.Theorem.Normal.Bundles.Subsection}

This final subsection is devoted to the proof of 
\begin{proposition} \label{Van.Thm.NB.Subsection.Statement}
Let $X \subseteq \PP^r$ be a smooth complex projective variety of dimension $n$, and denote by $N = N_{X/\PP^r}$ the normal bundle to $X$. Then
\[ \HHH{i}{X}{S^kN \otimes \det N \otimes \OO_X(\ell)} \ = \ 0 
\]
for all $i > 0$,  $k\ge 0$ and $\ell \ge -r$. 
\end{proposition}
\noi Here $\OO_X(k)$ denotes $\OO_{\PP^r}(k)|X$. We remark that similar statements were established by Schneider and Zintl in \cite{Schneider.Zintl}, but this particular vanishing does not seem to appear  there. Other vanishings for normal bundles played a central role in \cite{SAD}. 

\begin{proof} [Proof of Proposition \ref{Van.Thm.NB.Subsection.Statement}]
We use the abbreviation $\PP = \PP^r$. Starting from the  exact sequence $ 0 \lra TX \lra T\PP|X \lra N \lra 0$, we get a long exact sequence
\[ \ldots \lra S^{k-2}T\PP|X \otimes \Lambda^2 TX \lra S^{k-1}T\PP|X \otimes TX  \lra S^k T\PP|X \lra S^k N \lra 0. \tag{*}  \]
By adjunction, $\det N \otimes \OO_X(\ell) = \omega_X \otimes \OO_X(\ell+r+1)$.  So after twisting through by $\det N \otimes \OO_X(\ell)$ in (*), we see that the Proposition will follow if we prove:
\[
\HHH{i}{X}{S^{k-j}T\PP|X \otimes  \Lambda^j TX \otimes \omega_X \otimes \OO_X(\ell+ r + 1))} \ = \ 0 \  \ \ \text{for } \ i \, \ge \, j + 1 \tag{**}
\]
when $0 \le j \le k$ and  $\ell \ge -r$. It follows from the Euler sequence that $S^m T\PP|X$ has a presentation of the form
\[  0 \lra \oplus \, \OO_X(m-1) \lra \oplus \, \OO_X(m) \lra S^m T\PP |X \lra 0, \]
so for (**) it suffices in turn to verify that
\[ \HH{i}{X}{\Lambda^j TX \otimes \omega_X \otimes \OO_X(\ell_1)} \ = \ 0 \] for $i \ge j + 1$ and $\ell_1 > 0$. But 
$ \Lambda^j TX \otimes \omega_X = \Omega^{n-j}_X$,
so finally we're asking that
\[ \HH{i}{X}{\Omega^{n-j}_X \otimes \OO_X(\ell_1)} \ = \ 0 \ \ \text{for } \ i \ge j + 1  \ \text{and }   \ell _1 >0,\]
and this follows from  Nakano vanishing.   \end{proof}

\section{Proof of Theorem \ref{Intro.Sat.Deg.Thm} }

We now turn to the proof of Theorem \ref{Intro.Sat.Deg.Thm}
 from the Introduction.
 
Consider then a non-singular variety $X \subseteq  \PP^r$ that is cut out  as a scheme by hypersurfaces of degrees $d_0 \ge \ldots \ge d_p$. Equivalently, we are given a surjective homomorphism of sheaves:
\[ \eps : U \lra \II_X \ \ \text{,} \ \ U \ = \ \oplus \, \OO_{\PP^r}(-d_i). \]
Let 
$\mu : \PP^\pr = \Bl_X(\PP^r) \lra \PP^r  $ be the blowing up of $X$, with exceptional divisor $\EEE \subseteq \PP^\pr$, so that $\II_X \cdot \OO_{\PP^\pr} = \OO_{\PP^\pr} (- \EEE)$. Write $H$ for the pull-back to $\PP^\pr$ of the hyperplane class on $\PP^r$, and set $U^\pr = \mu^* U$. Thus on $\PP^\pr$ we have a surjective map of bundles:
\begin{equation} \label{Map.of.Bundles.on.Blowup}
\eps^\pr : U^\pr \lra \OO_{\PP^\pr}(-\EEE).
\end{equation}
Noting that
\[  \HHH{0}{\PP^\pr}{\OO_{\PP^\pr}( tH - a \EEE)} \ = \  \HHH{0}{\PP^r}{\II_X^a \otimes\OO_{\PP^r}(t)},
\]
one sees as in Lemma \ref{Surjectivity.Suffices.Lemma} that the question is to prove the surjectivity of 
\begin{equation} \label{Surjectivity.Required.for.Thm.A}
\HHH{0}{\PP^\pr}{S^aU^\pr\otimes \OO_{\PP^\pr}(tH)} \lra \HHH{0}{\PP^\pr}{\OO_{\PP^\pr}(tH- a \EEE)} 
\end{equation}
for $t \ge ad_0 + d_1 + \ldots + d_r - r$. 

To this end, we pass to the Buchsbaum--Eisenbud complex \eqref{Koszul.Complex.Power} constructed from \[ U^\pr \otimes \OO_{\PP^\pr}(\EEE)   \overset{\eps^\pr} \lra \OO_{\PP^\pr} \lra 0. \]
Twisting through by $\OO_{\PP^\pr}(t H - a\EEE)$, we arrive at a long exact sequence of vector bundles on $\PP^\pr$ having the form:

\vskip -10pt
\small
\begin{equation} \label{Big.Complex.on.Blowup}
\xymatrix@C=9.5pt@R=12pt{
\ldots \ar[r] &\Schur{a}{2} U^\pr \otimes \OO_{\PP^\pr}(t H + 2\EEE) \ar[r] \ar@{=}[d]&S^{a,1} U^\pr \otimes \OO_{\PP^\pr}(t H+ \EEE) \ar[r]\ar@{=}[d] & S^aU^\pr \otimes \OO_{\PP^\pr}(t H) \ar[r] \ar@{=}[d] &\OO_{\PP^\pr}(t H - a \EEE) \ar[r] &0. \\ & C_2 & C_1 & C_0
}
\end{equation}
\normalsize
With indexing as indicated, the $i^{\text{th}}$ term of this sequence is given by
\[
C_i \ = \ \Schur{a}{i}(U^\pr)   \otimes  \OO_{\PP^\pr}(tH + i \EEE). 
\]

In  order to establish the surjectivity \eqref{Surjectivity.Required.for.Thm.A} it suffices upon chasing through \eqref{Big.Complex.on.Blowup} to prove that
\begin{equation} \label{Vanishing.Required.for.Thm.A}
\HH{i}{\PP^\pr}{C_i} \ = \ 0 \ \ \text{ for } \ 1 \le i \le r
\end{equation}
provided that $t \ge ad_0 + d_1 + \ldots + d_r -r$. But now recall (Remark \ref{Properties.of.Schur.Power}) that if $i \le r$ then $\Schur{a}{i}(U^\pr)$ is a sum of line bundles $\OO_{\PP^\pr}(mH)$ with 
\[
m  \, \ge \, -ad_0 - d_i - \ldots - d_i \ge -ad_0 - d_1 - \ldots - d_r.
\]
Hence when $t \ge ad_0 + d_1 + \ldots + d_r -r$, $C_i$ is a sum terms of the form
\[  
\OO_{\PP^\pr}(\ell H + i \EEE)  \ \ \text{with } \ell \ge -r.
\]
Therefore \eqref{Vanishing.Required.for.Thm.A}
 -- and with it Theorem \ref{Intro.Sat.Deg.Thm} -- is a consequence of
\begin{proposition}
If $\ell \ge -r$, then
\[  \HH{i}{\PP^\pr}{\OO_{\PP^\pr}{(\ell H + i \EEE)} } \ = \ 0 \ \text{  for \ } i > 0.\]
\end{proposition}
\begin{proof}
Thanks to the Leray spectral sequence, it suffices to show:
\[
\HH{j}{\PP^r}{R^k \mu_* \OO_{\PP^\pr}(\ell H + i \EEE)} \ = \ 0 \ \ \text{when  } j + k = i > 0. \tag{*}
\]
For $ k = 0$, observe that $\mu_* \OO_{\PP^\pr}(\ell H + i \EEE) = \OO_{\PP^r}(\ell)$, and these sheaves have no higher cohomology when $\ell \ge -r$. On the other hand, by Proposition \ref{Pushforward.Ext.Equation} the only non-vanishing higher direct images are the $R^{e-1}\mu_* \OO_{\PP^\pr}(\ell H + i \EEE)$, which do not appear when $i \le e-1$. So (*) holds when $j = 0, k = e-1$. It remains to consider the case  $k = e-1$ and $i \ge e$, so $j = i - (e-1) > 0$.  
Here Corollary \ref{Filtration.of.Push.Forwards}
 implies that the $R^{e-1}$ have a filtration with quotients
 \[
 S^\alpha N \otimes \det N \otimes \OO_X(\ell),
 \]
 where as above $N = N_{X/\PP^r}$ is the normal bundle to $X$ in $\PP^r$. But since we are assuming $\ell \ge -r$,  Proposition \ref{Van.Thm.NB.Subsection.Statement} guarantees that these sheaves have vanishing higher cohomology. This completes the proof. \end{proof}

\begin{remark}\label{Few.Equations}
Observe that if $X$ is defined by $p < r$ equations, then the argument just completed goes through taking
  $  d_{p+1} = \ldots =  d_r = 0.$
  \end{remark}

 %
 %
 %
 %


\begin{thebibliography}{EMS}
 \setlength{\parskip}{3pt}
 
 \bibitem{AV} Alessandro Arsie and Jon Eivind Vatne, A note on symbolic and ordinary powers of ideals, \textit{Ann. \!Univ. \!Ferrara} \textbf{49} (2003), 19-30.
 
 \bibitem{BEL} Aaron Bertram, Lawrence Ein and Robert Lazarsfeld, Vanishing theorems, a theorem of Severi and the equations defining projective varieties, \textit{Journal of AMS} \textbf{4} (1991), 589-602.
 
 \bibitem{BocciHarbourne} Cristiano Bocci and Brian Harbourne, Comparing powers and symbolic powers of ideals, \textit{J. \!of Algebraic Geom.} \textbf{19} (2010), 399-417.
 
 \bibitem{Buchsbaum.Eisenbud} David Buchsbaum and David Eisenbud, Generic free resolutions and a family of generically perfect ideals, \textit{Advances in Math.} \textbf{18} (1975), 245-301.
 
 \bibitem{CM-S.P} James Carlson, Stefan Muller-Stach and Chris Petters, \textit{Period Mappings and Period Domains}, 2nd Edition, 
 Cambridge Studies in Advanced Mathematics, 2017.
 
 
 \bibitem{Caviglia} Giulio Caviglia, Kozul algebras, Castelnuovo--Mumford regularity, and generic initial ideals, PhD Thesis, University of Kansas, 2005.

\bibitem{Chandler} Karen Chandler, Regularity of the powers of an ideal, \textit{Comm. \!in Algebra} \textbf{25} (1997), 3773-3776.

\bibitem{Cooper+} Susan Cooper, Giuliana Fatabbi, Elena Guardo, Anna Lorenzini, Juan Migliori, Uwe Nagel, Alexendra Seceleanu, Justyna Szpond and Adam Van Tuyl, Symbolic powers of codimension two Cohen--Macaulay ideals, \textit{Communications in Algbra} \textbf{48} (2020), 4663-4680.

\bibitem{CHK} Dale Cutkosky, J\"urgen Herzog and Ng\^ o Vi\^et Trung, Asymptotic behavior of the Castelnuovo--Mumford regularity, \textit{Composition Math.} \textbf{118}  (1999), 243-261.

\bibitem{Dao.ea} Hailong Dao, Alessandro De Stefani, Eloisa Grifo, Craig Huneke and Luis N\'u\~nez-Betancourt, Symbolic powers of ideals, in \textit{Singularities and Foliations. Geometry, Topology and Applications}, Spriger Proceedings in Math and Statistics \textbf{222} (2018), 387-432.

\bibitem{SAD} Lawrence Ein and Robert Lazarsfeld, Syzygies and Koszul cohomology of smooth projective varieties of arbitrary dimension, \textit{Inventiones Math.} \textbf{111} (1993), 51-67.

\bibitem{ELS} Lawrence Ein, Robert Lazarsfeld and Karen Smith, Uniform bounds and symbolic powers on smooth varieties, \textit{Inventiones Math.} \textbf{144} (2001), 241-252.

\bibitem{Eisenbud.Goto} David Eisenbud and Shiro Goto, Linear free resolutions and minimal multiplicity, \textit{J of Algebra} \textbf{88} (1984), 89--133.

\bibitem{Fulton.Lang} William Fulton and Serge Lang, \textit{Riemann--Roch Algebra}, Grundlehren Math. \!Wiss. vol. \!277, Springer--Verlag 1985.

\bibitem{GGP} Anthony Geramita, Alessandro Gimigliano and Yves Pitteloud, Graded Betti numbers of some embedded rational $n$-folds, \textit{Math. \!Ann.} \textbf{301} (1995), 363-380.


\bibitem{GLP} Laurent Gruson, Robert Lazarsfeld and Christian Peskine, A theorem of Castelnuovo and the equations defining space curves, \textit{Inventiones Math.} \textbf{72} (1983), 491-506.


\bibitem{Hochster.Huneke} Melvin Hochster and Craig Huneke, Comparison of symbolic and ordinary powers of ideals, \textit{Inventiones Math.} \textbf{147} (2002), 249-369.


\bibitem{Huybrechts} Daniel Huybrechts, \textit{Fourier-Mukai Transforms in Algebraic Geometry}, Oxford Mathematical Monographs.   Oxford University Press, 2006

\bibitem{Kempf} George Kempf, The singularities of certain varieties in the Jacobian of a curve, PhD Thesis, Columbia University, 1971.

\bibitem{Kod} Vijay Kodiyalam, Asymptotic behavior of Castelnuovo--Mumford regularity, \textit{Proc. \!AMS} \textbf{128} (2000), 407-411.

  \bibitem{PAG} Robert Lazarsfeld, \textit{Positivity in Algebraic Geometry}, I \& II, Ergebnisse Math. \!vols. 48 \& 49, Springer Verlag, 2004. 



\bibitem{Schneider.Zintl} Michael Schneider and J\"org Zintl, The theorem of Barth--Lefschetz as a consequence of Le Potier's vanishing theorem, \texttt{manuscripta math.} \textbf{80} (1993), 259-263.

\bibitem{Sidman} Jessica Sidman, On the Castelnuovo--Mumford regularity of products of ideal sheaves, \textit{Advances in Geometry} \textbf{2} (2002), 219--229.


 \bibitem{Srinivasan} Hema Srinivasan, Algebra structures on some canonical resolutions, \textit{J. \!of Algebra} \textbf{122} (1989), 150-187.
 
 \bibitem{Tchernev} Alexandre Tchernev, Acycliciy of symmetric and exterior powers of complexes, \textit{J. \!of Algebra} \textbf{184} (1996), 1113-1135.
 
 
 \bibitem{Weyman} Jerzy Weyman, Resolutions of the exterior and symmetric powers of a module, \textit{J. \!of Algebra} \textbf{58} (1979), 333-341.

  \end{thebibliography}
 \end{document}